\documentclass[12pt,leqno]{article}

\usepackage{amssymb}
\usepackage{amsmath}
\usepackage{amsfonts}
\usepackage{graphicx}
\usepackage{pgf}
\usepackage{caption}
\usepackage{subcaption}
\usepackage{float}

\def\RR{{\mathbb{R}}}

\def\f12{{\scriptstyle{\frac{1}{2}}}}
\def\CC{{\mathbb{C}}}

\def\NN{{\mathbb{N}}}

\def\O{{\cal O}}

\def\S{{\cal S}}

\def\V{{\cal V}}

\def\X{{\cal X}}

\def\subneq{\mathop{\raise 0.4ex \hbox{~$\subset$}}
{\raise -0.4ex \hbox{$\!\!\!\!\!\scriptstyle{-}$}}
{\raise -0.2ex \hbox{$\!\!\!\scriptscriptstyle{/}$}}\ \,}

\def\eop{\hbox{{\vrule height7pt width3pt depth0pt}}}
\newcommand{\LL}{\mathop{\longrightarrow}}

\newcommand{\dm}[1]{{\displaystyle{#1}}}
\newcommand{\integ}[4]{\dm{\int_{#1}^{#2}\!{#3}\,d{#4}}}

\newcommand\bfor{\begin{eqnarray*}} 
\newcommand\efor{\end{eqnarray*}} 
\newcommand\bfornum{\begin{eqnarray}} 
\newcommand\efornum{\end{eqnarray}} 
\newcommand{\ba}{\begin{array}}
\newcommand{\ea}{\end{array}}

\newtheorem{thm}{Theorem}
\newtheorem{prop}{Proposition}

\newtheorem{lem}{Lemma}

\newenvironment{proof}
{\begin{sloppypar}\noindent{\sf Proof~: }}
{\hspace*{\fill}\eop
\medskip
\end{sloppypar}}
\title{An improvement of the   product integration method for a weakly singular Hammerstein equation }
\author{Laurence Grammont\thanks{Universit\'e de Lyon, Institut Camille Jordan, UMR 5208, 23 rue du Dr Paul Michelon, 
42023 Saint-\'Etienne Cedex 2, France. laurence.grammont@univ-st-etienne.fr}, Hanane Kaboul \thanks{Universit\'e de Lyon, Institut Camille Jordan, UMR 5208, 23 rue du Dr Paul Michelon, 
42023 Saint-\'Etienne Cedex 2, France. hanane.kaboul@univ-st-etienne.fr }}

\date{} 

\begin{document}

\maketitle 

\abstract{We present a new method to solve nonlinear Hammerstein equations with weakly singular kernels. The  process to approximate the solution,                                                 followed usually,  consists in adapting the discretization scheme from the linear case in order to obtain a nonlinear system in a finite dimensional space and solve it by any linearization method. In this paper,                                                                 we propose to first linearize, via                                                                                                                                                                                                                                                                                        Newton method, the nonlinear operator equation and only then to discretize  the obtained linear equations by the product integration method. We prove that  the iterates, issued from our method, tends to the exact solution of the nonlinear Hammerstein equation when the number of Newton  iterations tends to infinity, whatever the discretization parameter can be. This is not the case when the discretization is done first: in this case, the                            accuracy of the approximation is limited by the mesh                            size discretization. A Numerical example is given to confirm the theorical result.}

\bigskip


\bigskip


\section{Introduction}
The general framework of this paper is the following. 
Let $\X$ be a complex Banach space  and $K:\O\subseteq\X\to\X$   a nonlinear Fr\'echet differentiable integral operator  of the Hammerstein type  defined on a 
nonempty open set $\O$ of $\X$:
\begin{eqnarray}\label{defK}
K(x)(s):=\integ{a}{b}{ H(s,t)L(s,t){F(t,x(t))}}{t}, \   \mbox{ for all } x\in \Omega,
\end{eqnarray} 
where $H$ is the singular part of the kernel, $L$ is the regular part of the kernel and $F$, the nonlinear part of the operator, is a real-valued function of  two variables~: 
$$
(t,u)\in [a,b]{\times}\RR\mapsto F(t,u)\in\RR,
$$
with enough regularity so that $K$ is twice Fr\'echet-differentiable on $\O$. \\
The problem is  
  a nonlinear Fredholm integral equation of the second kind:
\begin{eqnarray}\label{problem}
\mbox{Find }\varphi\in\O~:\quad\varphi-K(\varphi)=y,
\end{eqnarray}
for a given function $y\in\X$.  \\   
Let $T:=K'$ denote the Fr\'echet derivative of $K$, i.e., for all $x\in\O$, 
\begin{eqnarray}\label{defT}
T(x)h(s)=\integ{a}{b}{H(s,t)L(s,t)\dfrac{\partial F}{\partial u}(t,x(t))h(t)}{t},\quad h\in\X,s\in[a,b].
\end{eqnarray}
In the following, $\X$ will be the space of the real valued continuous functions over a real interval $[a,b]$, $C^0([a,b],\RR)$, equiped with the supremum norm $\Vert.\Vert$. 

If we consider a singular kernel  such as  $H(s,t):= \log(\vert s-t\vert)$ or $\vert s-t\vert^{\alpha}$, $0<\alpha<1$, an approximation based on standard numerical integrations is a poor idea. The main idea of the product integration method is introduced by Atkinson for linear integral equations (\cite{atkin0},\cite{atkin4} and \cite{atkin}) and is motivated by Young \cite{young}. The product integration method consists in performing a piecewise polynomial linear interpolation of the smooth part of the kernel times the function involving the  unknown. This method is called product trapezoidal rule when the interpolation is linear.  The solution of a second kind Fredholm integral equation with weakly singular kernel is typically nonsmooth near the boundary of the domain of integration. In order to obtain a high order of convergence, taking into account the singular behaviour of the exact solution, polynomial spline on graded mesh are developped (among other authors, Brunner, Pedas, Vainikko, Schneider \cite{brunner}, \cite{pedas} and \cite{schneider}). In \cite{kaneko}, Kaneko, Noren and Xu discuss a standard product integration method with a general piecewise polynomial interpolation for weakly singular Hammerstein equation and indicate its superconvergence properties. Under particular assumptions on the right hand side $y$, on the function $F$ defining the nonlinearity  and on the regularity of the exact solution of (\ref{problem}), they give an error estimation involving $n$ the discretization paramater and the degree  $m$  of the piecewise polynomial interpolation. Hence the error depends on $n$ and $m$. 

 In the 1950's, the major theme in the domain of theoretical numerical analysis was the developement of general frameworks in the domain of  functional analysis to build and analyze numerical methods. A particular important contribution in this context is the
paper of Kantorovich \cite{kantorovich1} and later \cite{kantorovich2}. It proposes a generalization of Newton's method for solving nonlinear operator equations on Banach spaces. This idea is used everywhere when dealing with integral equations or partial differential equations. In \cite{anselone1}, chapter 6, Anselone studies the Newton method to approximate solutions of nonlinear equations $P(x)=0$ where $P$ is a nonlinear differentiable operator from a Banach space into itself. When dealing with the convergence of approximate solutions,  these are defined as the solution of $P_n(x)=0$ where $P_n$ is an approximation of $P$. The Newton method is then applied to the functional equation $P_n(x)=0$. The philosophy of most of the papers dealing with the numerical approximation of nonlinear integral operator equation consists in defining an approximate operator $P_n$ to $P$ and then apply Newton method (\cite{ansorge},\cite{atkin1}, \cite{atkin2}, \cite{atkin3}, \cite{brunner},\cite{dell},\cite{gram3}, \cite{gram4},\cite{kaneko},\cite{kra},\cite{kra1},\cite{pedas} and \cite{vain}).  

We propose to apply Newton's method directly to the  operator equation $P(x)=0$ and then to discretize the linear operator equations, issued from the Newton's iterations, by a product integration method. We will prove that the approximate iterates tend to the exact solution of the operator equation as the number of iterations tends to infinity. The important fact is that the convergence holds  whatever  the discretization parameter,  defining the size of the linear system to be solved, can be. As we do not need  the solution of each Newton iteration to be particularly accurate, we chose to apply the classical trapezoidal product integration method.

Section 2 is devoted to the description of  our method (linearization via Newton's method  followed by discretization by the product integration method). In Section 3, the convergence result is proved. In the last section, the classical method (discretization followed by linearization) is recalled and we compare it with our method through     a numerical example.

\section{Description of the new method  } 
To solve the problem  (\ref{problem}) for a given function $y$ in $C^0([a,b],\RR)$, we propose to first apply  the Newton method to the equation  $\varphi-K(\varphi)=y$. It leads to the sequence $(\varphi^{(k)})_{k \geq 0} \in \O$: 
\begin{equation}\label{lin}
\varphi^{(0)} \in \O, \quad (I-T(\varphi^{(k)}))(\varphi^{(k+1)}-\varphi^{(k)})=-\varphi^{(k)}+K(\varphi^{(k)})+y, \ k\geq 0.
\end{equation}
Then we discretize this equation with the product integration method associated to the piecewise linear interpolation. \\
Let $\Delta_n$, defined by
\begin{equation}\label{deltan}
a=:t_{n,0}<t_{n,1}<\dots <t_{n,n}:=b,
\end{equation}
be the uniform grid of  $[a,b]$ with mesh  $h_n:=\dfrac{b-a}{n}$.\\

Setting 
$$f_x(t):=\dfrac{\partial F}{\partial u}(t,x(t )), $$
$[L(s,t)f_x(t)h(t)]_n$ denotes the piecewise linear interpolation of $L(s,t)f_x(t)h(t)$~:\\
 $\forall s\in [a,b],\forall i=1,\dots,n$:
 \begin{eqnarray*} 
 [L(s,t)f_x(t)h(t)]_n&:=&\frac{1}{h_n}  (t_{n,i}-t)L(s,t_{n,i-1})  f_x(t_{n,i-1} )h(t_{n,i-1}) \\
 && + \frac{1}{h_n}(t-t_{n,i-1})L(s,t_{n,i}) f_x(t_{n,i}))h(t_{n,i})  
\end{eqnarray*}
for  $t\in[t_{n,i-1},t_{n,i}]$.\\
 We define the approximate operator $T_n$ by
 \begin{eqnarray}\label{prodint} 
T_n(x)(h)(s):=\integ{a}{b}{H(s,t)[L(s,t) f_x(t)h(t)]_n}{t},\quad x\in\O, h\in\X,s\in[a,b].  
\end{eqnarray}

Then the approximate problem is:  
\begin{eqnarray}\label{intpro} 
\mbox{Find } \varphi^{(k+1)}_{n} \in X:(I-T_n(\varphi_n^{(k)}))(\varphi_n^{(k+1)}-  \varphi_n^{(k)})  =
  -\varphi_n^{(k)}+K(\varphi_n^{(k)})+y.
\end{eqnarray}
We have
\begin{eqnarray*}
T_n(x)(h)(s)
&:=&\sum_{j=0}^{n}w_{n,j}(s)L(s,t_{n,j})f_x(t_{n,j})h(t_{n,j}),\\
w_{n,0}(s)&:=& \frac{1}{h_n}\int_{t_{n,0}}^{t_{n,1}}H(s,t)(t_{n,1}-t)dt,  \\
w_{n,n}(s)&:=& \frac{1}{h_n}\int_{t_{n,n-1}}^{t_{n,n}}H(s,t)(t-t_{n,n-1})dt,\\
\mbox{and for} \ 1\leq j\leq n-1,\\
w_{n,j}(s)&:=&\frac{1}{h_n}\int_{t_{n,j-1}}^{t_{n,j}}H(s,t)(t-t_{n,j-1})dt+\frac{1}{h_n}\int_{t_{n,j}}^{t_{n,j+1}}H(s,t)(t_{n,j+1}-t)dt. 
\end{eqnarray*}
Hence (\ref{intpro}) can be rewritten as 
 
\begin{eqnarray}\label{nonlinapp}
& &\varphi_n^{(k+1)}(s)-
\sum_{j=0}^{n}w_{n,j}(s)L(s,t_{n,j})f_k(t_{n,j})\varphi_n^{(k+1)}(t_{n,j})\\
 &=& K(\varphi_n^{(k)})(s)+y(s)-\sum_{j=0}^{n}w_{n,j}(s)L(s,t_{n,j})f_k(t_{n,j})\varphi_n^{(k)}(t_{n,j}),
\end{eqnarray}
where 
$$f_k(t_{n,j}):=\dfrac{\partial F}{\partial u}(t_{n,j},\varphi_n^{(k)}(t_{n,j})).$$

Setting 
$$
x_n^{(k+1)}(j):=\varphi_n^{(k+1)}(t_{n,j}),
$$

From the evaluations of equation (\ref{nonlinapp}) at the nodes of the grid,  it is straightforward that the vector $x_n^{(k+1)}$ is the solution of the linear system
\begin{equation}
(I-A_n^{(k)})x_n^{(k+1)}=b_n^{(k)},
\end{equation}
where
\begin{eqnarray*} 
A_n^{(k)}(i,j)&:=& w_{n,j}(t_{n,i})L(t_{n,i},t_{n,j})f_k(t_{n,j}), \\
b_n^{(k)}&:=& K(\varphi_n^{(k)})(t_{n,i})+y(t_{n,i})- A_n^{(k)}x_n^{(k)}.
\end{eqnarray*} 
 $\varphi_n^{(k+1)}$ is recovered from equation (\ref{nonlinapp}) ~: 

\begin{eqnarray*} 
\varphi_n^{(k+1)}(s)&=& \sum_{j=1}^{n}w_{n,j}(s)L(s,t_{n,j})f_k(t_{n,j})\left( x_n^{(k+1)}(t_{n,j})-x_n^{(k)}(t_{n,j})\right) \\
 &+& K(\varphi_n^{(k)})(s)+y(s).   
\end{eqnarray*}

\section{Convergence property of the new method} 

Existence, uniqueness and regularity properties of the solution  of equation (\ref{problem}) have been already considered (for example by Kaneko, Noren and Xu \cite{kaneko2} or  Pedas and Vainikko   \cite{pedas2}). In this section, we are only interested by the proof of the convergence of 
$\varphi_n^{(k)}$ towards $\varphi$ when $k\rightarrow \infty$.  

\noindent \textbf{Hypotheses:}

\begin{description}
\item [(H0)] $F$, defined in (\ref{defK}), is twice continuously differentiable on $[a,b]\times \RR$.
\item[(H1)]\label{h1}$L\in C^0([a,b]\times [a,b],\RR).$
\item[(H2)]$H$ verifies:
\begin{description}
\item[(H2.1)]$\displaystyle{c_H:=\sup\limits_{s\in [a,b]}\int_a^b|H(s,t)|dt<+\infty.}$
\item[(H2.2)]$\displaystyle{\lim\limits_{h\to 0}\omega_H(h)= 0,}$ where$$
\omega_H(h):=\sup\limits_{|s-\tau |\leq h,\:s,\tau\in[a,b]}\int_a^b|H(s,t)-H(\tau,t)|dt.
$$
\end{description}
\end{description}
\begin{description}
\item[(H3)] $\varphi \in \O$  is an isolated solution of $\varphi-K(\varphi)=f$.
\end{description}
\begin{description}
\item[(H4)]$I-T(\varphi)$ is invertible.
 \end{description}
 
 These assumptions ensure that equations (\ref{problem}),(\ref{intpro}) and (\ref{nonlinapp}) for $n$ large enough, are uniquely solvable (see \cite{atkin}).\\
 Let $a>0$ such that $ B(\varphi,a)\subset \O $, where $B(\varphi,a)$ denotes the open ball in $C^0([a,b],\RR)$ centered in $\varphi$ and of radius $a$. As $\varphi \in  \subset C^0([a,b],\RR)$, $\varphi$ is bounded. As $F$ is twice continuously differentiable, the following constant exists: \begin{eqnarray*}
 &M_2(a)&:=\sup_{t \in [a,b], \vert u\vert\leq a+\Vert \varphi\Vert} \vert \displaystyle\frac{\partial^2 F}{\partial u^2}(t,u)\vert.
 \end{eqnarray*}

The proof of  convergence  relies on  the successive approximations convergence result  (see \cite{kra1} Theorem 2.3. pp 21). Let us recall this result (in a form slicely different from \cite{kra1}).  

\begin{prop}\label{pro1} Consider a nonlinear operator $A$ from a Banach space $\X$ into itself, defined on an open set $\V$. Let $x^* \in \V$ be a  fixed point of $A$. 
Let the operator $A$ be Fr\'echet differentiable at the point $x^*$.
Let us assume that the following condition is fulfilled 
\begin{equation}\label{hyp}
 \rho_0:=\rho(A'(x^*)) < 1,
\end{equation}
where  $\rho$ denotes the spectral radius  and $A'$ denotes  the Fr\'echet derivative of $A$. 

Then, for all $\varepsilon >0$ such that  $\rho_0+\varepsilon< 1$, there exist   $r_{\varepsilon}>0$ and $r'_{\varepsilon}>0$ such that $B(x^*,r_{\varepsilon}) \subset \V$ and $B(x^*,r'_{\varepsilon}) \subset \V$ and such that for   $x^{(0)}$ in $B(x^*,r'_{\varepsilon})$, the successive approximations $x^{(k)}$ defined by
$$
x^{(k+1)}=A(x^{(k)}), \ k\in \NN, 
$$
remain in $B(x^*,r_{\varepsilon})$ for all $k\in \NN$, and the sequence $(x^{(k)})_{k\geq 0}$ converges to $x^*$. Moreover,   
\begin{equation*}
\Vert x^{(k)}-x^* \Vert \leq  r_{\varepsilon} (\rho_0+\varepsilon)^k.
\end{equation*}
\end{prop}

The following four lemmas are needed to prove our main result.

\begin{lem}  
For all $a>0$ such that $ B(\varphi,a)\subset \O $,
  for all $x\in B(\varphi,a)$,
 \begin{eqnarray}\label{Tlip} 
\Vert T(x)-T(\varphi)\Vert \leq c_H c_L  M_2(a) \Vert x-\varphi\Vert,
 \end{eqnarray} 
  where$$
c_L:=\max_{s,t\in [a,b]}|L(s,t)|.
$$
\end{lem} 
\begin{proof}
 From the  mean value theorem  applied to $\dfrac{\partial F}{\partial u}$, for all $t$, there exist  a real number $c(t)$ between $\varphi(t)$ and $x(t)$ such that
$$  
\dfrac{\partial F}{\partial u}(t,x(t))-\dfrac{\partial F}{\partial u}(t,\varphi(t))=\dfrac{\partial^2 F}{\partial u^2}(t,c(t)) (x(t)-\varphi(t)).  
$$
As $\forall t \in [a,b], \vert c(t)\vert \leq r+\Vert \varphi  \Vert$, 
\begin{eqnarray}\label{Tlip} 
\Vert T(x)-T(\varphi)\Vert &\leq& \Vert x-\varphi \Vert \sup\limits_{s\in [a,b]}\int_a^b \vert H(s,t)\vert \vert L(s,t)\vert \vert \dfrac{\partial^2 F}{\partial u^2}(t,c(t))\vert dt \\
&\leq&  c_H c_L  M_2(a) \Vert x-\varphi\Vert.
 \end{eqnarray} 
Hence (\ref{Tlip}) is  deduced. 
\end{proof}
Now $a$ is fixed such that $B(\varphi,a)\subset \O$. 

\begin{lem} \label{lem2}
 There is a positive number $r<a$ such that for all $x \in B(\varphi,r)$, $I-T(x)$ is invertible and  
\begin{eqnarray} 
\Vert (I-T(x))^{-1}\Vert \leq 2\mu, 
\end{eqnarray}
where $\mu:=\Vert (I-T(\varphi))^{-1}\Vert$.\\   Moreover, for all $x \in B(\varphi,r)$,  for $n$ large enough, $I-T_n(x)$ is invertible and  there exists a constant $c_x>0,$ independent of $n$,  such that
 \begin{eqnarray} 
 \Vert (I-T_n(x))^{-1}\Vert \leq  c_x. 
 \end{eqnarray} 
\end{lem}
\begin{proof}

Let $0<r<a$ be such that   $  r      \leq \displaystyle\frac{1}{2\mu c_H c_L M_2(a)}.$
 
For all $x\in B(\varphi,r)$,
$$
I-T(x)=I-T(\varphi)+T(\varphi)-T(x)=(I-T(\varphi))[I+(I-T(\varphi))^{-1}( T(\varphi)-T(x))].
$$

Since 
$$
\|(I-T(\varphi))^{-1}(T(\varphi)-T(x))\|\le \mu
 \|T(\varphi)-T(x)\|\le\mu c_H c_L   M_2(a)r    
     \le\frac{1}{2},
$$
we conclude that $I-T(x)$ is invertible and that its inverse is uniformly bounded on $B(\varphi,r)$. In fact
$$
(I-T(x))^{-1}=\big[ I+(I-T(\varphi))^{-1}( T(\varphi)-T(x))\big]^{-1}(I-T(\varphi))^{-1},
$$
so
$$
\|(I-T(x))^{-1}\|\le\mu\sum\limits_{k=0}^\infty\|(I-T(\varphi))^{-1}( T(\varphi)-T(x))\|^k\le2\mu.
$$
The function $(s,t)\rightarrow L(s,t)\dfrac{\partial F}{\partial u}(t,x(t))$ is in $C^0([a,b]\times [a,b],\RR)$. Hence,
according to \cite{atkin} or \cite{anselone1},   $T_n(x) \LL\limits^{p} T(x)$,
 where $\LL\limits^{p}$ denotes the  pointwise convergence, and the sequence $(T_n(x))_{n\geq 0}$ is collectively compact.

For all $x \in B(\varphi,r)$, $T_n(x)$ is a collectively compact approximation of  $T(x)$ (see \cite{anselone1}). Hence for $n$ large enough, $I-T_n(x)$ is invertible and $(I-T_n(x))^{-1}$ is uniformly bounded in $n$. This means that there is a constant $c_x$ such that for $n$ large enough, $$
\Vert (I-T_n(x))^{-1}\Vert \leq c_x.
$$
This ends the proof.  
 
\end{proof}

 Let $A_n$ be the operator defined on   $B(\varphi,r)$   by 
 \begin{eqnarray}\label{A}
 A_n(x):=x+S_n(x)\left( K(x)+y-x\right),
\end{eqnarray} 
where 
\begin{eqnarray}\label{S}
 S_n(x):=(I-T_n(x))^{-1}.
\end{eqnarray} 
Notice that we have 
\begin{eqnarray}\label{ptfixean}
A_n(\varphi)=\varphi.
\end{eqnarray} 

\begin{lem}\label{lem3}
The operator $S_n$ is Fr\'echet differentiable at  $\varphi$.  
\end{lem}
  \begin{proof}
 As for all  $h\in
 \X$,  $T_n(x)(h)(s)=\displaystyle\sum_{j=0}^{n} w_{n,j}(s)L(s,t_{n,j})\displaystyle\frac{\partial F}{\partial u}(t_{n,j},x(t_{n,j}))h(t_{n,j})$,  the operator $x \mapsto T_n(x)$ is differentiable at $\varphi$, and we have for all  $h\in
 \X$ and $\delta \in \X$, 
$$T'_n(x)(\delta,h)(s)=\sum_{j=0}^{n}w_{n,j}(s) L(s,t_{n,j}) \displaystyle\frac{\partial^2 F}{\partial u^2}(t_{n,j},x(t_{n,j}))\delta(t_{n,j})  h(t_{n,j}).$$
As, at the first order, $(L+E)^{-1}-(L)^{-1}\backsimeq -(L)^{-1}E(L)^{-1} $, we have
\begin{eqnarray*}
S_n(\varphi+\delta)-S_n(\varphi)&=&(I-T_n(\varphi+\delta))^{-1}-(I-T_n(\varphi))^{-1}\\
&=&(I-T_n(\varphi)-T_n'(\varphi)\delta+O(\delta^2)))^{-1}-(I-T_n(\varphi))^{-1}\\
&\backsimeq& S_n(\varphi)T_n'(\varphi)\delta S_n(\varphi), \\
\end{eqnarray*}
so that $S_n$ is Fr\'echet differentiable at  $\varphi$ and
\begin{eqnarray*}
S'_n(\varphi)\delta= S_n(\varphi)T_n'(\varphi)\delta S_n(\varphi).
\end{eqnarray*}
  \end{proof}
\begin{lem}\label{lem4}
The operator $A_n$ is Fr\'echet differentiable at $\varphi$. For $n$ large enough, 
\begin{eqnarray}  
\rho(A'_n(\varphi)) < 1.
\end{eqnarray}
\end{lem}
\begin{proof}
Notice that  $A_n(\varphi)= \varphi$, hence 
 \begin{eqnarray*}  
A_n(\varphi+h)-A_n(\varphi)&=& \varphi+h+ S_n(\varphi+h)\left( K(\varphi+h)+y-\varphi-h\right)-\varphi\\
&=& h+S_n(\varphi+h)\left( K(\varphi)+T(\varphi)h+O(h^2)+y-\varphi-h\right)\\
&=& h+S_n(\varphi+h)\left(  (T(\varphi)-I)h+O(h^2)  \right),
\end{eqnarray*} 
hence $A_n$ is differentiable at $\varphi$ and 
\begin{eqnarray}\label{A'}  
A'_n(\varphi)=  I-S_n(\varphi)(I-T(\varphi)).
\end{eqnarray}  
 We have 
 \begin{eqnarray}
 \rho\left( I-S_n(\varphi)(I-T(\varphi)) \right) &=&\inf_n \Vert 
 (I-\S_n(\varphi)(I-T(\varphi)))^n\Vert^{ \frac{1}{n}} \\
 &\leq& \Vert 
 (I-\S_n(\varphi)(I-T(\varphi)))^2\Vert^{ \frac{1}{2}}. 
 \end{eqnarray}
Since $\left(I-\S_n(\varphi)(I-T(\varphi))\right)
 =S_n(\varphi) \left(T(\varphi)-T_n(\varphi)\right)$,
 \begin{align*}
\Vert \displaystyle\left(I-S_n(\varphi)(I-T(\varphi))\right)^2 \Vert &= \Vert S_n(\varphi)\left(T(\varphi)-T_n(\varphi)\right)S_n(\varphi)\left(T(\varphi)-T_n(\varphi)\right)\Vert \\
 & \leq  \Vert S_n(\varphi)\Vert \Vert \left(T(\varphi)-T_n(\varphi)\right)S_n(\varphi)\left(T(\varphi)-T_n(\varphi)\right)\Vert\\
 & \leq  c_{\varphi} \Vert \left(T(\varphi)-T_n(\varphi)\right)S_n(\varphi)\left(T(\varphi)-T_n(\varphi)\right)\Vert.
 \end{align*}
 As $S_n(\varphi)$ is uniformly bounded 
 (see Lemma  \ref{lem2}) and ($T(\varphi)-T_n(\varphi)$) is collectively compact, the closure of the  set 
 $ S:=\displaystyle\cup_n \{ S_n(\varphi)(T(\varphi)-T_n(\varphi))x, \Vert x \Vert \leq 1 \}$ is compact so that 
 \begin{align*}
\Vert \displaystyle\left(I-S_n(\varphi)(I-T(\varphi)\right)^2 \Vert  
 & \leq  c_{\varphi} \sup_{x \in S} \Vert \left(T(\varphi)-T_n(\varphi)\right)x\Vert \LL\limits_{n \rightarrow \infty}  0.
 \end{align*}
 Then 
 \begin{eqnarray*}  
\rho(A'_n(\varphi)) \LL\limits_{n \rightarrow \infty}  0.
\end{eqnarray*}
 
This ends the proof.  

\end{proof}

\begin{thm}\label{mainresult} 
Under assumptions {\bf (H0) to (H4)}, there exists $r>0$ such that, for a fixed $n$ large enough to have 
$$
\rho_n:=\rho  \left( A'_n(\varphi)\right) < 1,
$$ and for any $\varepsilon >0$ such that $\rho_n+\varepsilon < 1$, there exist 
$\  B(\varphi, r'_{n,\varepsilon})\subset B(\varphi, r)$ and $B(\varphi,r_{n,\varepsilon})\subset B(\varphi, r)$ such that, 
if $\varphi_n^{(0)} \in B(\varphi, r'_{n,\varepsilon})$, then the sequence $(\varphi_n^{(k)})_{k\geq 0}$ solution of 
\begin{eqnarray*} 
(I-T_n(\varphi_n^{(k)}))(\varphi_n^{(k+1)}-  \varphi_n^{(k)})  =
  -\varphi_n^{(k)}+K(\varphi_n^{(k)})+f,
\end{eqnarray*}
is defined, belongs to  $B(\varphi, r_{n,\varepsilon})$
    and 
$$
\varphi_n^{(k)}\LL\limits_{k\to\infty}  \varphi.
$$    
Moreover, the following estimation holds:  
\begin{equation} 
\Vert \varphi_n^{(k)}-\varphi \Vert \leq r_{n,\varepsilon} \left( \rho_n+\varepsilon\right)^k.
\end{equation}   
\end{thm}
\begin{proof}
The exact solution $\varphi$ of the nonlinear problem (\ref{problem}) is a fixed point of $A_n$ (see (\ref{ptfixean})). From Lemma \ref{lem4}, $A_n$ is Fr\'echet differentiable at $\varphi$ and for $n$ large enough $\rho(A_n'(\varphi))<1$. The conditions needed to apply Proposition \ref{pro1} to the operator $A_n$ are satisfied so that  
  Proposition \ref{pro1} gives the result.
\end{proof}

\section{Numerical evidence}

\subsection{The classical product integration method}
Let us recall the classical product integration method applied to nonlinear operators (see\cite{kaneko}).
The classical product integration approximation    $\psi_n$ solves the nonlinear equation 
\begin{equation*}
\psi_n(s)-\integ{a}{b}{H(s,t)[L(s,t)F(t,\psi_n(t))]_n}{t}=y(s),
\end{equation*} 
where $[L(s,t)F(t,\psi_n(t))]_n$ denotes the piecewise linear interpolant of $t\mapsto L(s,t)F(t,\psi_n(t))$. Using the uniform grid 
$\Delta_n$, we obtain
 
\begin{equation}\label{prodintcla}
\psi_n(s)-\sum_{j=0}^{n} w_{n,j}(s) L(s,t_{n,j})F(t_{n,j},\psi_n(t_{n,j})) =y(s),
\end{equation} 
where 
\begin{eqnarray*}
w_{n,0}(s)&:=& \frac{1}{h_n} \int_{t_{n,0}}^{t_{n,1}}H(s,t)(t_{n,1}-t) dt,\\
w_{n,j}(s)&:=& \frac{1}{h_n} \int_{t_{n,j-1}}^{t_{n,j}}H(s,t)(t-t_{n,j-1}) dt + \frac{1}{h_n} \int_{t_{n,j}}^{t_{n,j+1}}H(s,t)(t_{n,j+1}-t)  dt, \ j=1,\ldots,n-1,\\
w_{n,n}(s)&:=& \frac{1}{h_n} \int_{t_{n,n-1}}^{t_{n,n}}  H(s,t)(t-t_{n,n-1}) dt. 
\end{eqnarray*}
 
Set
\begin{eqnarray*}
{\sf Y}_n(i)&:=& y(t_{n,i}),\\
{\sf A}_n(i,j)&:=&=w_{n,j}(t_{n,i})L(t_{n,i},t_{n,j}),\\
{\sf X}_n&:=&\left[\begin{array}{c}
\psi_n(t_{n,0})\\
\vdots\\
\psi_n(t_{n,n})
\end{array}\right],\\
{\sf F}({\sf X}_n)&:=&\left[\begin{array}{c}
F(t_{n,0},\psi_n(t_{n,0}))\\
\vdots\\
F(t_{n,n},\psi_n(t_{n,n}))
\end{array}\right].
\end{eqnarray*}

Evaluating the equation (\ref{prodintcla}) at the nodes $t_{n,i}, \ i=0,\ldots,n$, the following non linear system is obtained: 

\begin{equation}\label{nonlinapp}
{\sf X}_n-{\sf A}_n{\sf F}({\sf X}_n)={\sf Y}_n,
\end{equation}
which we can solve by the classical finite dimensional Newton method (\cite{argybook} and \cite{argy}). Notice that the Newton's iterates  tends to $\psi_n$ when $k\rightarrow \infty$   and not to $\varphi$, the exact solution. It means that the accuracy of the approximate solution is limited by $n$. 

\subsection{Numerical Illustration}

Numerical experiments are now carried out to illustrate the accuracy  of our   method. Let us consider in $\X:=\CC^0([0,1],\RR)$, the operator
$$
K(\varphi)(s):=\integ{0}{1}{\kappa(s,t,\varphi(t))}{t},\quad \varphi\in\O\subseteq\X,\ s\in[0,1],
$$
with the real valued kernel function $\kappa$~: 
$$
(s,t,u)\in[0,1]{\times}[0,1]{\times}\RR\mapsto\kappa(s,t,u):=\log(|s-t|)\sin(\pi u)  .
$$
The exact solution of $\varphi(s)-K(\varphi)(s)=y(s)$ is 
$$
\varphi(s):=1
$$ 
for 
$$
y(s):=1.
$$

\noindent{\it Implementation remark}~:
To solve the  linear system at each Newton iteration, the integral $\displaystyle\int_a^blog(|s-t|)\sin(\varphi_n^{(k)}(t))dt$ needs to be evaluated. To evaluate it, we use the singularity subtraction technique (\cite{anselone2} and also \cite{ahularlim}). \\\\

In the following, our method is called  the linearization-discretization method and the classical one is called  the discretization-linearization method. We compare them.

 \begin{figure}[H]\label{fig1}
 \centering
\includegraphics[width=13truecm,height=8truecm]{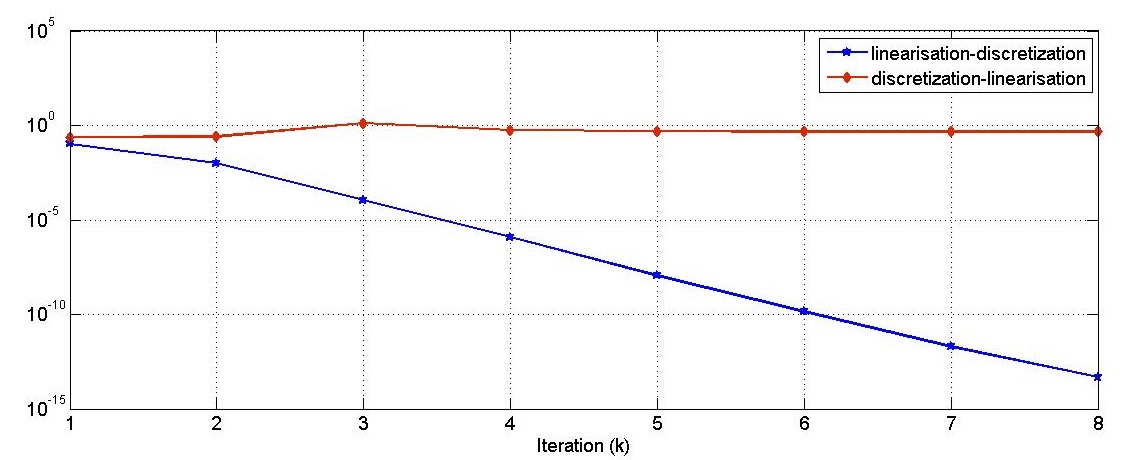} 
  \caption{Comparison of the   errors for $n=50$}
 \end{figure}
 
 The classical discretization-linearization method can not be accurate if it is performed on a coarse grid ($n=50$). We obtain the approximate solution $\psi_n$ of $\varphi$ and the error is constant with the number of iterations. On the other hand, our linearization-discretization method  approaches the exact solution $\varphi$ even if the discretization is done with a coarse grid. It confirms the theoretical result.   
 
 \begin{figure}[H]
 \label{fig2}
 \centering
 \includegraphics[width=13truecm,height=8truecm]{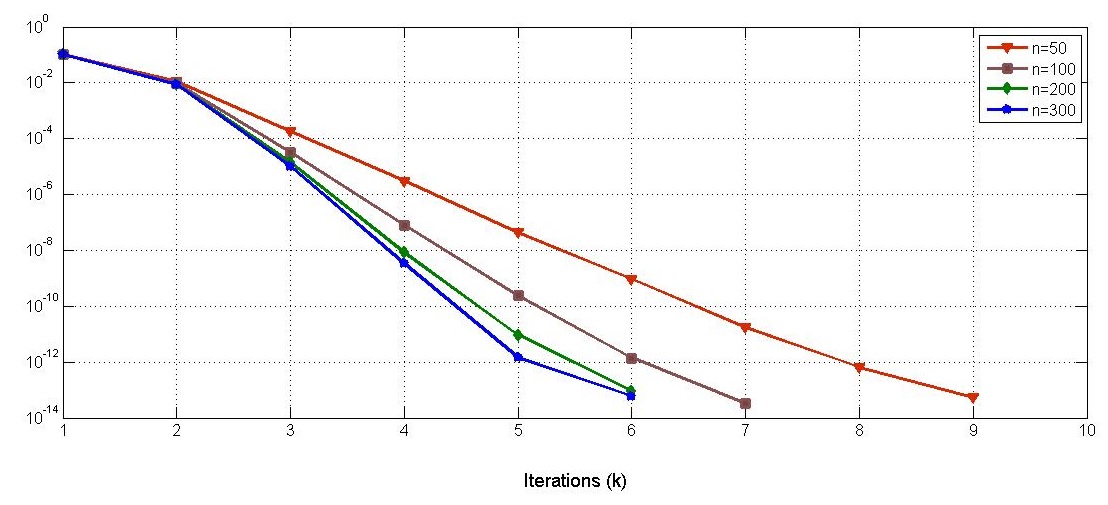}
  \caption{errors of the linearization-discretization method for increasing values of $n$}
 \end{figure}
 
  When $n$ increases, the number of Newton's iterations needed to reach a fixed accuracy decreases. It can be useful to apply higher order product integration methods  (see the idea of Diogo, Franco and Lima in \cite{diogo}), instead of the product trapezoidal rule,  to  reduce the  number of Newton's iterations and therefore the computational cost of the method. An high order product integration approximation, adapted to a Fredholm equation,   can also be a good starting point for the Newton's iterations of our method.

\begin{center}
{\bf Acknowledgements}
\end{center}
The first  author is partially                                                           supported by the Indo-French Centre for Applied Mathematics (IFCAM).

\end{document}